\documentclass[10pt]{amsart}

\usepackage[T1]{fontenc}
\usepackage{amssymb,amsmath,amsthm}
\usepackage{color,hyperref}
\usepackage{bbm}
\usepackage{enumitem}

\title[Bifurcation into spectral gaps]{%
  Bifurcation into spectral gaps for a noncompact semilinear 
  Schr\"odinger equation with nonconvex potential}
\author{Troestler C.}
\address{Institut de Math\'ematique\\
  Universit\'e de Mons\\
  Place du Parc 20\\
  B-7000 Mons (Belgium)\\
}
\email{Christophe.Troestler@umons.ac.be}


\newtheoremstyle{theorem}
{10pt} 
{10pt} 
{\slshape} 
{\parindent} 
{\bfseries} 
{. } 
{ } 
{} 
\theoremstyle{theorem}
\newtheorem{theorem}{Theorem}
\newtheorem{corollary}[theorem]{Corollary}
\newtheorem{prop}[theorem]{Proposition}
\newtheorem{lem}[theorem]{Lemma}

\newtheoremstyle{defi}
{10pt} 
{10pt} 
{\rmfamily} 
{\parindent} 
{\bfseries} 
{. } 
{ } 
{} 
\theoremstyle{defi}
\newtheorem{rem}[theorem]{Remark}

\makeatletter

\newcommand{\2}{2^\star}

\newcommand{\Dom}{%
  \mathop{\mathstrut
    \mathbb D}\nolimits}

\newcommand{\IZ}{{\mathbb Z}}
\newcommand{\IR}{{\mathbb R}}
\newcommand{\IC}{{\mathbb C}}
\newcommand{\IB}{{\mathrm B}}
\newcommand{\C}{{\mathcal C}}
\newcommand{\Leb}{{\mathrm L}}
\newcommand{\wto}{\rightharpoonup}
\newcommand{\id}{{\mathbbm 1}}
\let\subset=\subseteq 
\let\phi=\varphi
\let\liminf=\varliminf
\let\limsup=\varlimsup
\@ifundefined{leqslant}{}{\let\le=\leqslant}
\@ifundefined{geqslant}{}{\let\ge=\geqslant}
\newcommand{\intd}{{\,\mathrm d}}
\newcommand{\restrict}{\mathclose\upharpoonright}
\newcommand{\Order}{{\mathrm O}}
\newcommand{\order}{{\mathrm o}}
\newcommand{\D}{{\C^\infty_{\text{c}}}}
\newcommand{\Q}{{\mathcal Q}}
\newcommand{\E}{{\mathcal E}}

\newcommand{\loc}{{\mathrm{loc}}}
\newcommand{\<}{\langle}
\renewcommand{\>}{\rangle}
\DeclareMathOperator{\cl}{cl}

\definecolor{mygreen}{rgb}{0.2,0.47,0.2}
\definecolor{myblue}{rgb}{0.23,0.33,0.95}
\definecolor{myred}{rgb}{0.67,0.03,0.06}

\@ifundefined{hypersetup}{}{%
  \hypersetup{%
    pagebackref=true,
    pdfauthor={Christophe Troestler},
    pdftitle={Bifurcation into spectral gaps for a noncompact semilinear 
      Schrodinger equation with nonconvex potential},
    pdfsubject={bifurcation in spectral gaps},
    pdfkeywords={bifurcation, noncompact problems},
    colorlinks=true,
    urlcolor=myblue,
    citecolor=myblue,
    linkcolor=myred,
    pdfpagemode=UseNone,
    pdfstartview=FitH,
  }}

\makeatother

\begin{document}
\begin{abstract}
  \sloppy
  This paper shows that the nonlinear 
  periodic eigenvalue problem
  \begin{equation*}
    \begin{cases}
      -\Delta u + V(x) u - f(x,u) = \lambda u,\\
      u \in H^1(\IR^N),
    \end{cases}    
  \end{equation*}
  has a nontrivial branch of solutions emanating from
  the upper bound of every spectral gap of $-\Delta + V$.
  No convexity condition is assumed.
  The following result of independent interest is also proven:
  the direct sum $Y \oplus Z$
  in $H^1(\IR^N)$ associated to a decomposition of the
  spectrum of $-\Delta+V$ remains ``topologically direct'' in
  the $\Leb^p$'s (in the sense that the
  projections from $Y+Z$ onto $Y$ and $Z$ are  $\Leb^p$-continuous).
\end{abstract}

\maketitle

\section*{Introduction}

The purpose of this paper is to show that the nonlinear 
periodic eigenvalue problem
\begin{equation}
  \label{DiffEqn}    
  \begin{cases}
    -\Delta u + V(x) u - f(x,u) = \lambda u,\\
    u \in H^1(\IR^N),
  \end{cases}
\end{equation}
with $V$, $f$ being $\IZ^N$-periodic in $x$ and $f$ being superquadratic
but subcritical, has nontrivial branches of solutions bifurcating 
from the upper bound of every spectral gap of $-\Delta + V$ on 
$\Leb^2(\IR^N)$.

From now on, let us consider one of these spectral gaps, say $(a,b)
\subset \rho(-\Delta+V)$.  Of course, it is no lack of generality to
suppose that $0 \in (a,b)$.
This article is inspired from a previous one with M.~Willem~\cite{trch} 
where it is proven that (\ref{DiffEqn}) possesses a 
nontrivial solution for every $\lambda \in (a,b)$.  This approach was
subsequently refined by A.~Szulkin~\cite{Szulkin} who was able to pass
from $f \in \C^1$ to $f\in \C^0$.  
Here we will consider slightly weaker assumptions than 
in~\cite{Szulkin}, namely
\begin{list}{}{\topsep0ex \itemsep.3ex}
\item[$(f1)$]
  $V\in \Leb^\infty(\IR^N)$ and $f\in\C^0(\IR^N\times \IR)$ are
  1-periodic in $x_k$, $1\le k \le N$, and the linear operator
  $$
    D:\Leb^2(\IR^N) \to \Leb^2(\IR^N) : u \mapsto -\Delta u+V(x) u$$
  with domain $\Dom(D) = H^2(\IR^N)$ is  invertible 
  (with continuous inverse);
\item[$(f2)$]
  there exists $2< p<\2 := 2N/(N-2)$ and $c>0$ such that for 
  all $(x,u) \in \IR^N\times\IR$~:
  $|f(x,u)| \le c(1 + |u|^{p-1})$;
\item[$(f3)$]
  $f(x,u) = \order(|u|)$ uniformly in $x\in \IR^N$ as $u \to 0$;
\item[$(f4)$]
  there exists $\alpha >2$ such that~: for every $u\in\IR$ 
  and every   $x \in \IR^N$,
  $ 0 \le  \alpha F(x,u) \le f(x,u)u$;
\item[$(f5)$]
  $\displaystyle
  \liminf_{|u| \to \infty} \,
  \mathop{\min\vphantom{\liminf}}\limits_{x \in [0,1]^N} F(x,u) >0$.
\end{list}
where $F(x,u) := \int_0^u f(x,v) \intd v$.
The proofs given in~\cite{Szulkin} are still valid under
$(f1)$--$(f5)$ and so there exists a nontrivial
solution $u_\lambda$ of~(\ref{DiffEqn}) for all $\lambda \in (a,b)$.
(See section~\ref{section-bifurcation} for more details.)
This solution is obtained as a critical point of
$$
  \E_\lambda : H^1(\IR^N) \to \IR : u \mapsto
  \tfrac{1}{2} \int_{\IR^N} |\nabla u|^2 + (V-\lambda) u^2
  - \int_{\IR^N} F(x,u) \intd x.$$
Indeed, under $(f1)$--$(f5)$, $\E_\lambda$ is well defined on 
$H^1(\IR^N)$ and possesses the linking geometry 
(see section~\ref{section-bifurcation}).
The main improvement of~\cite{trch,Szulkin} with respect to
previous works on~(\ref{DiffEqn}) 
(see~\cite{Charles} and the references therein) 
is the removal of any convexity
condition upon $F$.  However in 
the latter, it was proved that
$u_\lambda$ actually bifurcates from $(\lambda, u) = (b,0)$;
and so a question raises itself: does this remain true
for nonconvex $F$'s?  The question is here settled positively
under the additional assumption:
\begin{list}{}{\topsep0ex \itemsep.3ex}
\item[$(f6)$]
  there exists a nonnegative $\IZ^N$-periodic function 
  $B \in \Leb^\infty(\IR^N)\setminus\{0\}$ and $\beta < \2$ such that
  $F(x,u) \ge B(x) |u|^\beta$ for all $x\in \IR^N$ and all
  $u$ in a neighborhood of $0$;
\end{list}
which, together with $(f4)$, may be seen as a local
``pinching condition''.  A global one was used
in~\cite{Charles} (see condition $(P)$, p.~20).
Note that  $(f4)$ implies $\beta \ge \alpha$.
Actually, since a possible $B(x)$ is
$\min\{ \liminf_{u \to 0} F(x,u) |u|^{-\beta}, 1\}$
and $F$ is periodic, $(f6)$ means that the set of $x \in [0,1]^N$
satisfying $\liminf_{u \to 0} F(x,u) |u|^{-\beta} >0$ has 
nonzero measure.
The main theorem of this paper reads as follows.

\penalty-1000
\begin{theorem}
\label{mainresult}
  Let $(f1)$--$(f6)$ hold  and $(a,b)$ be the spectral gap of $D$
  containing~$0$.  Then, for each $\lambda \in (a,b)$, there exists a
  nontrivial solution $u_\lambda$ of~(\ref{DiffEqn}) such that
  $$
    \E_\lambda(u_\lambda) 
    = \Order\bigl( (b-\lambda)^{\beta/(\beta-2) - N/2} \bigr) \to 0
    \quad\hbox{as } \lambda \to b.$$
  Furthermore, if $\beta < 2 + 4/N$,
  $$
    \| u_\lambda \| = 
    \Order\bigl( (b - \lambda)^{1/(\beta -2) - N/4} \bigr) \to 0,$$
  where $\|\cdot\|$ denotes the usual norm on $H^1(\IR^N)$.
\end{theorem}

The above condition on $\beta$ is optimum in the sense that, if
$F(x,u) = |u|^\gamma$, then $\alpha \le \gamma \le \beta$  so that
the best choice for $\beta$ is $\beta := \gamma$; but $\gamma < 2 + 4/N$
is necessary for a bifurcation to take place at $b$ when 
$b = \inf \sigma(D)$
(see~\cite{Charles5}).
Moreover, as observed by T.~K\"upper and C.A.~Stuart 
\cite{Charles4,Charles3}, no bifurcation  can occur at $a$.
But of course, if $f$ is such that $-f$ satisfies $(f1)$--$(f6)$, 
a branch of nontrivial solutions bifurcates from $(\lambda,u) = (a,0)$
with the convergence rates of theorem~\ref{mainresult}
(change the signs of $\lambda$ and $u$ to recover the initial problem).

Before going, in section~\ref{section-bifurcation}, through the
estimates from which bifurcation will eventually result, some 
preliminary discussion about the spectral properties of the quadratic
part $\Q_\lambda$ of $\E_\lambda$ is necessary.  It is carried out in
section~\ref{section-spectral-gap}.
In particular, it is said that $H^1(\IR^N)$ splits as a  direct sum
of two closed subspaces $Y$ and $Z$ on which $\Q_\lambda$ is negative
and positive definite respectively.  It is of great importance
for the projection from $Y+Z$ onto $Y$
(or $Z$) to be continuous in the $\Leb^p$'s---and not only in $H^1$.
This is not the case for every direct sum in $H^1$.  However
appendix~\ref{section-spectral-Lp} shows that it is true for the
particular sum associated to the positive and negative part of the
spectrum of $-\Delta + V$.
We chose to expound this in an appendix not to interrupt the arguments
about bifurcation.

Some natural questions are left unanswered by this paper.
First, it would be interesting to know whether the bifurcating
branch is continuous.  Second, as we  said, there is no 
bifurcation at $a$.  But does any  nontrivial solution go to
$\infty$ as $\lambda \to a$?

\noindent\textit{Notations.}
We will write $|u|_p$ for the norm of $u$ in 
the Lebesgue space $\Leb^p(\IR^N)$,
$(\cdot|\cdot)_2$ for the inner product in $\Leb^2(\IR^N)$, 
$\|\cdot\|$ for the usual norm on the Sobolev space $H^1(\IR^N)$,
$\partial \mathcal{F}(u)$ will stand for the Fr\'echet derivative of
the function $\mathcal{F}$ at~$u$,
$\Dom(A)$ for the domain of the operator $A$, and
$\IB(x,R)$ will denote the open ball in $\IR^N$ with center $x$
and radius $R$.

\section{The quadratic form and Bloch waves}
\label{section-spectral-gap}

Let $\Q_\lambda : H^1(\IR^N) \to \IR$ be the quadratic form
$$
  \Q_\lambda(u) := \int_{\IR^N} 
               |\nabla u|^2 + (V(x) -\lambda) u^2 \intd x.$$
Since $0$ lies in a gap of the spectrum $\sigma(D)$,
spectral theory asserts that $H^1(\IR^N)$ splits as a direct sum
of two closed subspaces
$Y$ and $Z$ on which $\Q_0$ is negative and positive definite
respectively:
\begin{equation}
  \Q_0(y) \le -\alpha_0 \|y\|^2,
  \qquad
  \Q_0(z) \ge \beta_0 \|z\|^2
  \label{form-sign}
\end{equation}
for all $y\in Y$ and $z\in Z$.  Moreover, $Y$ and $Z$ are 
orthogonal in $\Leb^2(\IR^N)$, $\Q_0(y+z) = \Q_0(y) + \Q_0(z)$, and
the spectral gap is $(a,b)$ with
\begin{equation}
  a:= \sup_{\textstyle{y \in Y \atop |y|_2 =1}}  \Q_0(y)
  < 0 <
  \inf_{\textstyle{z \in Z \atop |z|_2 =1}} \Q_0(z) =: b.
  \label{spectral-gap}
\end{equation}
The same spectral splitting holds for any $\lambda \in (a,b)$.
This is made precise by the following lemma.

\begin{lem}
  \label{SplittingQ}
  Let $\lambda \in (a,b)$.  Then
  $$
    \Q_\lambda(y) \le -\alpha_\lambda \|y\|^2
    \quad\hbox{and}\quad
    \Q_\lambda(z) \ge \beta_\lambda \|z\|^2$$
  for all $y\in Y$ and $z \in Z$, where
  $$
    \alpha_\lambda :=
    \begin{cases}
      \alpha_0 (1- \lambda/a)&  \text{if } \lambda\le 0,\\
      \alpha_0&                 \text{otherwise},
    \end{cases}
    \qquad
    \beta_\lambda :=
    \begin{cases}
      \beta_0&                  \text{if } \lambda\le 0,\\
      \beta_0(1-\lambda/b)&     \text{otherwise.}
    \end{cases}
    $$
  Consequently, $\Q_\lambda(z) - \Q_\lambda(y) \ge  N_\lambda \|y+z\|^2$
  with $N_\lambda := {1\over 2} \min\{ \alpha_\lambda, \beta_\lambda \}$.
\end{lem}

\begin{proof}
  We only deal with $\Q_\lambda$ on $Y$, the proof on $Z$ being similar.
  If $\lambda >0$, $\Q_\lambda(y) \le \Q_0(y) \le -\alpha_0 \|y\|^2$.
  If $\lambda \le 0$, 
  \begin{equation*}
    \Q_\lambda(y) =   \Q_0(y) - \lambda |y|_2^2  
                      \le \Q_0(y) - (\lambda/a) \Q_0(y)
                  \le - (1- \lambda/a) \alpha_0 \|y\|^2 .
    \qedhere
  \end{equation*}
\end{proof}

Since $b \in \sigma(D)$,
we know that there exists a Bloch wave $\Psi$ in
$H^2_\loc(\IR^N) \cap \C^1(\IR^N) \cap \Leb^\infty(\IR^N)$ that satisfies 
$-\Delta \Psi +V \Psi = b\Psi$
(see~\cite{Eastham}).  For $R \in (0,+\infty)$, let us set
$$
  \Psi_R (x) :=
        R^{-N/2}\, \eta(x/R)\, \Psi(x)$$    
where $\eta \in \D(\IR^N;[0,1])$ equals $1$ on $\IB(0,1)$.
Using the fact that $\Psi$ is uniformly almost-periodic
in the sense of Besicovich~\cite{Besicovich}, we get the 
following (see~\cite{Charles2,Charles}):
\vskip0pt
  {\openup2\jot
  \halign to\textwidth{\hbox to3em{$#$\hfil}& $\displaystyle #$\hfil
                  \tabskip0pt plus1000pt\crcr
  (B1)&  \Psi_R \in H^2(\IR^N) \cap \C^1(\IR^N);  \cr
  (B2)&  \limsup_{R \to \infty} \|\Psi_R\| < +\infty ;\cr
  (B3)&  \lim_{R \to \infty}\, R^2
         \int_{\IR^N}  |\nabla \Psi_R|^2 + (V -b) \Psi_R^2  
         \in [0, +\infty);  \cr
  (B4)&  \lim_{R \to \infty}\, R^2\,
         \bigl| -\Delta \Psi_R + (V-b) \Psi_R \bigr|_2^2
         \in [0,+\infty);  \cr
  (B5)&  \lim_{R \to \infty}\, R^{(\gamma -2) N/2}\,
         \int_{\IR^N} B(x) |\Psi_R |^\gamma \intd x  \in (0,+\infty)
         \hbox{ for all } \gamma \in [1,+\infty);  \cr
  (B6)&  | \Psi_R |_\infty = \Order(R^{-N/2} ).\cr }}%
The following consequences of $(B3)$--$(B4)$  will be used 
in place of them:
\vskip0pt
  {\openup2\jot
  \halign to\textwidth{\hbox to3em{$#$\hfil}& $\displaystyle #$\hfil
                \tabskip=0pt plus1000pt\crcr
  (B'3)&  \Q_b(\Psi_R) = \Order\bigl( R^{-2} \bigr)
          \hbox{ as }R \to \infty;\cr
  (B'4)&  \| \partial\Q_b(\Psi_R)\|^2 = \Order\bigl( R^{-2} \bigr)
          \hbox{ as } R \to \infty.\cr}}%
\vskip-1\jot
Let $P$ be the projector onto $Y$ and $Q = \id - P$ the projector 
onto $Z$.  For $\lambda \in (a,b)$, let us define
$$
  \zeta_\lambda := Q \Psi_{R(\lambda)} \in Z$$
with $R(\lambda) := (b-\lambda)^{-1/2}$.
The following holds:

\begin{lem}
\label{ZetaEstimations}
  When $\lambda \to b$, we have, for all $\gamma \in [2,\2]$,
  $$
    \limsup \|\zeta_\lambda\| < +\infty;  
    \qquad
    \Q_\lambda(\zeta_\lambda) = \Order\bigl( b-\lambda \bigr);  $$
  $$
    \liminf{} (b-\lambda)^{-(\gamma-2)N/4} 
       \int_{\IR^N} B(x) |\zeta_\lambda|^\gamma \intd x > 0;    $$
  $$
    \zeta_\lambda \in \Leb^\infty(\IR^N)
    \hbox{ and\/{} }
    | \zeta_\lambda |_\infty 
    =\Order\bigl( (b-\lambda)^{N/4} \bigr).$$
\end{lem}
  
\begin{proof}
  Since $R \mapsto \Psi_R$ is bounded near $\infty$ and $Q$ is continuous,
  $\lambda \mapsto \zeta_\lambda$ is bounded near~$b$.
  
  All along the rest of this proof, we will write $R$ for $R(\lambda)$.
  When $\lambda$ is close to $b$, it follows from the coercivity of
  $-\Q_\lambda$ on $Y$ that
  $$
    2 \alpha_0 \|P \Psi_R\|^2
     \le -2 \Q_\lambda(P\Psi_R) 
      =  - \bigl\< \partial\Q_\lambda(\Psi_R), P\Psi_R \bigr\> 
     \le \| \partial\Q_\lambda(\Psi_R)\| \, \|P\Psi_R\| 
     $$
  and so  $\|P\Psi_R\| = \Order\bigl( \|\partial \Q_\lambda(\Psi_R)\| \bigr)$.  
  But   $\partial\Q_\lambda(\Psi_R) = \partial\Q_b(\Psi_R) + (b-\lambda)
  \Order(|\Psi_R|_2)$. Thus, using $(B2)$, we infer
  $\partial\Q_\lambda(\Psi_R) = \Order(b-\lambda)$ and
  $$
    \|P\Psi_R\| = \Order(b-\lambda)
    \quad\hbox{as }\lambda \to b.$$
  The second estimate follows from
  $$
    \Q_\lambda(\zeta_\lambda) = 
    \Q_\lambda(\Psi_R - P\Psi_R) =
    \Q_b(\Psi_R) + (b-\lambda) |\Psi_R|_2^2 - \Q_\lambda(P\Psi_R) =
    \Order(b-\lambda).$$
  As for the third one, it is suffices to note 
  $\bigl( \int B\, |\zeta_\lambda|^\gamma \bigr)^{1/\gamma}
   \ge \bigl| \bigl( \int B\, |\Psi_R|^\gamma \bigr)^{1/\gamma} 
   - \bigl( \int B |P\Psi_R|^\gamma \bigr)^{1/\gamma} \bigr|$ and to use
  $(B5)$ and
  \begin{math}
    \int B\, |P\Psi_R|^\gamma
    = \Order(\|P\Psi_R\|^\gamma)
    = \Order\bigl((b-\lambda)^\gamma\bigr)
    =  \order\bigl( (b-\lambda)^{(\gamma-2)N/4} \bigr)
  \end{math}.

  Finally, the last assertion follows from
  proposition~\ref{LebCountinuityProj} (appendix~\ref{section-spectral-Lp}) 
  and~$(B6)$.
  Indeed $\Psi_R \in H^1(\IR^N) \cap \Leb^\infty(\IR^N)$ and
  the restriction of $Q$ to $H^1 \cap \Leb^\infty$ ranges in
  $H^1 \cap \Leb^\infty$ and is $\Leb^\infty$-continuous.
\end{proof}

\section{Bifurcation}
\label{section-bifurcation}

First of all, we shall explain the minimax construction
that gives a critical point $u_\lambda$ for all $\lambda \in (a,b)$.
Let us define the minimax value:
$$
  c_\lambda := \mathop{\inf\vphantom{\sup}}_{t\ge 0\mathstrut}\,
               \sup_{\eta_\lambda(t,M_\lambda)}   \E_\lambda$$
where $\eta_\lambda(t,u)$ is the flow generated by some pseudogradient
vector field approximating $-\nabla\E_\lambda$ 
(see~\cite{Szulkin,trch}) and let $M_\lambda$ be the set
$$
  M_\lambda := \bigl\{
  y+ s \zeta_\lambda : y\in Y,\ s\ge 0,\ 
  \|y+s\zeta_\lambda\| \le \rho_\lambda  \bigr\}$$
with $\rho_\lambda$ large enough such as $\sup_{\partial M_\lambda} \E_\lambda <0$
where $\partial M_\lambda$ is the boundary of $M_\lambda$ in 
$Y \oplus \IR \zeta_\lambda$.
Under slightly stronger assumptions than $(f1)$--$(f5)$,
it is proven in~\cite{Szulkin} that there
exists a Palais-Smale sequence at level $c_\lambda$ and that, for
any such Palais-Smale sequence $(u_n)$, there exists a sequence 
of translations $(k_n) \subset \IZ^N$ such that 
$\bigl( u_n(\cdot -k_n) \bigr)_n$ possesses a subsequence that
weakly converges to a nonzero critical point of $\E_\lambda$.

This conclusion remains valid under $(f1)$--$(f5)$.
Let us quickly explain why.  First, we keep having~(1.7) of~\cite{Szulkin}
that reads
$$
  \forall \delta >0,\ \exists c_1 >0,\quad
  F(x,u) \ge c_1 |u|^\alpha - \delta |u|^2 
  \hbox{ on } \IR^N \times \IR.$$
Indeed, for large $u$'s, say $|u| \ge \rho$, $(f4)$ and $(f5)$ imply
that
$$
  F(x,u) \ge c_2 |u|^\alpha$$
and $c_1$ can be taken small enough so that $c_1 \le c_2$ and\
$ c_1 |u|^\alpha - \delta |u|^2  \le 0 \le F(x,u)$
for all $(x,u) \in \IR^N \times [-\rho, \rho]$.
As a consequence, $\E_\lambda$ possesses the so called ``linking
geometry''.   The existence of a $(PS)_{c_\lambda}$-sequence
then follows---for this part relies only on the above geometry and the weak
continuity of $u \mapsto \partial\E_\lambda(u)$.
Finally, inequalities~(1.10) and~(1.11) of~\cite{Szulkin} need
not $f(x,u)u$ to be positive but only nonnegative---see 
eq.~(\ref{Estim2}) below.  So, any $(PS)_{c_\lambda}$-sequence
contains a subsequence that weakly converges, up to translations,
to a nonzero critical point.

For all $\lambda\in (a,b)$, let $u_\lambda \neq 0$ be such a limit
point.
We will show that $u_\lambda$ bifurcates from $(\lambda, u) = (b,0)$.
Let us start with some estimates of the energy of $u_\lambda$.

\begin{prop}
\label{clambdaEstimations}
  Let assumptions $(f1)$--$(f5)$ hold.  Then
  \begin{enumerate}
  \item\label{bound-E} $0 \le \E_\lambda(u_\lambda) \le c_\lambda$.
  \end{enumerate}
  If in addition  $(f6)$ is assumed and $\beta < 2 + 4/N$, we have
  \begin{enumerate}[resume]
  \item\label{bound-c_lambda}
    $c_\lambda = \Order\bigl( (b-\lambda)^{\beta/(\beta-2) - N/2} \bigr)
    \to 0$ as $\lambda \xrightarrow{<} b$.
  \end{enumerate}
\end{prop}

\begin{proof}
  \textsl{(\ref{bound-E})}  Let $\lambda$ be fixed.  
  Since $\E_\lambda(u) - {1\over 2} \partial\E_\lambda(u)u
  = \int {1\over2}f(x,u) - F(x,u) 
  \ge ({\alpha\over 2} -1) \int F(x,u)$,
  it is clear that any critical point of $\E_\lambda$ occurs 
  at a nonnegative level.
  
  Let $(u_n)$ be a Palais-Smale sequence at level $c_\lambda$ such that
  $u_n \wto u_\lambda$ in $H^1(\IR^N)$.
  The limit $u_\lambda$ is a critical point of $\E_\lambda$.
  Let us define
  $$
    \mu_\infty :=
    \lim_{R \to \infty} \, \limsup_{n \to \infty}
    \int\limits_{|x| > R}
    {\textstyle{1\over 2}}f(x,u_n)u_n - F(x,u_n) \intd x.$$
  It is clear that $\mu_\infty \ge 0$ and moreover, taking in account
  that $H^1(\IR^N)$ is compactly embedded in all $\Leb^r_\loc(\IR^N)$
  for $2 < r < \2$, one readily proves that (see e.g.~\cite{trch2}):
  \begin{equation*}
    \lim_{n \to \infty} \int_{\IR^N}
      {\textstyle{1\over 2}} f(x,u_n)u_n - F(x,u_n) \intd x
    = \int_{\IR^N}  {\textstyle{1\over 2}} 
      f(x,u_\lambda)u_\lambda - F(x,u_\lambda) \intd x + \mu_\infty.    
  \end{equation*}
  This can be rewritten as
  $$
    c_\lambda
      = \E_\lambda(u_n) - {\textstyle{1\over 2}} 
        \<\partial\E_\lambda(u_n), u_n \> + \order(1)
      = \E_\lambda(u_\lambda) - {\textstyle{1\over 2}} 
        \<\partial\E_\lambda(u_\lambda), u_\lambda \> + \mu_\infty,
  $$
  which implies $c_\lambda \ge \E_\lambda(u_\lambda)$.

  \begin{sloppypar}
  \textsl{(\ref{bound-c_lambda})}
  It follows from the very definition of $c_\lambda$ that it is
  bounded above by $\sup\bigl\{ \E_\lambda(y+s\zeta_\lambda) :
  y \in Y,\ s \ge 0 \bigr\}$.  
  Assumption~$(f6)$ tells us that there exists some $r>0$ such that
  $$
    F(x,u) \ge B(x) |u|^\beta
    \quad\hbox{for all $|u|\le r$ and $x \in \IR^N$.}$$
  Now  $(f4)$ says that $u \mapsto  F(x,u) |u|^{-\alpha}$ 
  is nondecreasing on $[0, +\infty)$ and nonincreasing on $(-\infty, 0]$,
  and so
  $$
    F(x,u) \ge F(x, \pm r) r^{-\alpha} |u|^\alpha
           \ge \kappa_1 B(x) |u|^\alpha
    \quad\hbox{for all $|u|\ge r$ and $x \in \IR^N$,}$$
  where $\kappa_1 := r^{\beta -\alpha}$.  
  Consequently,
  $$
    F(x,u) \ge \kappa_2 B(x) \min\{ |u|^\beta, |u|^\alpha \}
    \quad\hbox{for all }(x,u) \in \IR^N \times \IR$$
  with $\kappa_2 := \min\{ 1, \kappa_1 \}$.  
  Let $H$ a convex function given by 
  lemma~\ref{ConvexLowerBound} (appendix~\ref{Existenceof H}).
  Then,
  \begin{equation}
    \E_\lambda(y + s \zeta_\lambda)
    \le {\textstyle{1\over 2}} \Q_\lambda(y) 
    +   {\textstyle{1 \over 2}} s^2 \Q_\lambda(\zeta_\lambda)
    -   \kappa_2 \phi(y + s\zeta_\lambda)
    \label{PhiLower}    
  \end{equation}
  where $\phi(u) := \int B(x)  H(u) \intd x$.
  Using lemma~\ref{ConvexLowerBound} \ref{convex-H}
  and~\ref{pseudo-homog}, we infer that,
  for all $u,v\in H^1(\IR^N)$,
  $
    \phi(u+v) \le 
    {\textstyle {1\over 2}} \bigl( \phi(2u) + \phi(2v) \bigr)
    \le 2^{\beta-1} \bigl( \phi(u) + \phi(v) \bigr)$
  and consequently, for all $u,w$,
  $$
    \phi(u+w) \ge 2^{1-\beta} \phi(u) - \phi(w)$$
  (remember $\phi$ is even).
  Lemma~\ref{ConvexLowerBound} \ref{pseudo-homog}
  together with this inequality imply
  \begin{align}
    \phi(y+s\zeta_\lambda)
    &\ge \min\{ s^\beta, s^\alpha \} \,\phi(\zeta_\lambda + y/s)
    \notag\\
    &\ge \min\{ s^\beta, s^\alpha \} 
    \bigl( 2^{1-\beta} \phi(\zeta_\lambda) - \phi(y/s) \bigr).
    \label{PhiConvex}
  \end{align}
  On the other hand,
  since $\sup_{y \in Y,\, s\ge 0} \E_\lambda(y+s\zeta_\lambda) > 0$,
  we may as well just take the supremum on the $(y,s)$'s that
  satisfy $\E_\lambda(y+s\zeta_\lambda) \ge 0$ and $s > 0$.
  Using sucessively $F \ge 0$, lemma~\ref{SplittingQ}, and
  $\phi(u) \le |B|_\infty \min\{ |u|^\beta_\beta, |u|^\alpha_\alpha \}
  \le \kappa_3 \min_{\gamma=\alpha, \beta} \|u\|^\gamma$, we get
  \begin{align}
    \E_\lambda(y + s\zeta_\lambda) \ge 0
    &\quad\Rightarrow\quad
    \Q_\lambda(y) + s^2 \Q_\lambda(\zeta_\lambda) \ge 0  \notag\\
    &\quad\Rightarrow\quad
    \| y/s \|^2 \le \Q_\lambda(\zeta_\lambda)/\alpha_\lambda   \notag\\
    &\quad\Rightarrow\quad
    \phi(y/s) \le \kappa_3 \min_{\gamma=\alpha, \beta} 
               \bigl( \Q_\lambda(\zeta_\lambda)
                      /\alpha_\lambda \bigr)^{\gamma/2} .
     \label{orderOny/s}    
  \end{align}
  Taking account of $\Q_\lambda(y) \le 0$ and 
  (\ref{PhiLower})--
  (\ref{orderOny/s}), we get
  \begin{align}
    \E_\lambda(y + s \zeta_\lambda)
    &\le {\textstyle{1 \over 2}} s^2 \Q_\lambda(\zeta_\lambda)
          -   \kappa_2  \min\{ s^\beta, s^\alpha \} 
              \bigl( 2^{1-\beta} \phi(\zeta_\lambda) 
                     - \phi(y/s) \bigr)  \notag\\
    &\le \max_{\gamma=\alpha, \beta}\,
         {\textstyle{1 \over 2}} s^2 \Q_\lambda(\zeta_\lambda)
          -  \kappa_2 s^\gamma \Phi_\lambda  \notag\\
    \intertext{where}
    \Phi_\lambda &:= 2^{1-\beta} \phi(\zeta_\lambda) 
             - \kappa_3 \min_{\gamma=\alpha, \beta} 
               \bigl( \Q_\lambda(\zeta_\lambda)
                      /\alpha_\lambda \bigr)^{\gamma/2}   \notag
    \intertext{and thus, provided that $\Phi_\lambda >0$
      (see below),}
    \E_\lambda(y + s \zeta_\lambda)
    &\le \kappa_4 \max_{\gamma=\alpha, \beta}
         \Q_\lambda(\zeta_\lambda)^{\gamma/(\gamma-2)}\,
         \Phi_\lambda^{-2/(\gamma-2)}  ,
      \label{UpperBoundOnE} 
  \end{align}
  and $\kappa_4 := \max\bigl\{  ({1\over 2}-{1\over\gamma})  
  (\gamma \kappa_2)^{-2/(\gamma -2)} : \gamma=\alpha, \beta   \bigr\}$.
  \end{sloppypar}

  Let $\lambda \xrightarrow{<} b$.
  Then $\alpha_\lambda = \alpha_0$ and 
  $\Q_\lambda(\zeta_\lambda) = \Order(b-\lambda)$, so that
  \begin{equation}
    \min_{\gamma = \alpha, \beta}
    \bigl( \Q_\lambda(\zeta_\lambda)/\alpha_\lambda \bigr)^{\gamma/2}
    = \Order\bigl( (b-\lambda)^{\beta/2} \bigr).
    \label{termy/s}    
  \end{equation}
  By lemma~\ref{ZetaEstimations}, $|\zeta_\lambda|_\infty$ is 
  bounded, say by $r$.
  Lemma~\ref{ConvexLowerBound}~\ref{H-near0} implies there exists some
  $\kappa_5 > 0$ such that $H(u) \ge \kappa_5 |u|^\beta$ for
  $|u| \le r$.
  Consequently, one can infer
  $$
    \phi(\zeta_\lambda) 
    \ge  \kappa_5  \int B\, |\zeta_\lambda|^\beta
    \ge  \kappa_6  (b-\lambda)^{(\beta-2)N/4}$$
  which, together with~(\ref{termy/s}), yields
  $$
    \liminf_{\lambda \to b}{}
    (b-\lambda)^{-(\beta-2) N/4}\, \Phi_\lambda  >0,$$
  because 
  $(b-\lambda)^{\beta/2} = \order\bigl( (b-\lambda)^{(\beta-2) N/4} \bigr)$.
  This, incidentally, implies that $\Phi_\lambda >0$.
  It then suffices to plug the estimates of 
  $\Q_\lambda(\zeta_\lambda)$ and $\Phi_\lambda$ in 
  equation~(\ref{UpperBoundOnE}) to obtain
  \begin{equation*}
    c_\lambda = \Order\Bigl( \max_{\gamma=\alpha,\beta}
    (b-\lambda)^{(\gamma - \delta)/(\gamma-2)}    \Bigr)
    \qquad\text{where } \delta := \tfrac{1}{2}N(\beta-2).
  \end{equation*}
  To get the desired result, simply note that
  $\alpha/(\alpha-2) - \tfrac{1}{2}N (\beta-2)/(\alpha-2)
   \ge \beta/(\beta-2) -N/2 >0 $ 
  whenever $\beta < 2 + 4/N$.
\end{proof}

\begin{theorem}
  Assume $(f1)$--$(f6)$.  Then, when $\lambda \to b$,
  $ \|u_\lambda\| = \Order(\sqrt{c_\lambda/ N_\lambda})$, and in particular
  $$
    \|u_\lambda\| = \Order\bigl( (b-\lambda)^{1/(\beta-2)-N/4} \bigr)
    \to 0$$
  if $\beta < 2 + 4/N$.
\end{theorem}
\penalty-1000

\begin{rem}
  1) $N_\lambda$ is defined in lemma~\ref{SplittingQ}.

  2) The above conclusion is slightly stronger than the one of theorem~9.6
  in~\cite{Charles}.  The latter indeed states (under some additional
  assumptions) that $\lim_{n \to \infty} (b-\lambda_n)^{-\theta} \|u_n\|
  =0$ for all $0 \le \theta < 1/(\beta-2) - N/4$, where $\lambda_n < b$ is
  a suitable sequence converging to $b$, and $u_n$ is a critical point 
  of $\E_{\lambda_n}$.
\end{rem}

\begin{proof}
  Using proposition~\ref{clambdaEstimations} and $(f4)$,  we infer
  \begin{align}
    c_\lambda
    \ge \E_\lambda(u_\lambda)
    &=   \E_\lambda(u_\lambda) - {\textstyle{1\over2}} 
        \< \partial\E_\lambda(u_\lambda), u_\lambda \>      
    =  \int  {\textstyle{1\over2}} f(x,u_\lambda)u_\lambda
        - F(x,u_\lambda)  \intd x                       \notag\\
    &\ge {\textstyle ({1\over 2}-{1\over \alpha})}
         \int f(x,u_\lambda) u_\lambda \intd x         
    \label{Estim1}          
  \end{align}
  \sloppy%
  Assumptions $(f2)$--$(f3)$ imply the existence of a constant
  $\kappa_1$ such that $|f(x,u)| \le \kappa_1 |u|$ if $|u| \le 1$
  and $|f(x,u)| \le (\kappa_1 |u|)^{p-1}$ if $|u| \ge 1$.
  Because $f(x,u)u \ge0$, this yields
  \begin{equation}
    f(x,u)u = |f(x,u)| \, |u| \ge
    \begin{cases}
      \kappa_1^{-1} |f(x,u)|^2&  \text{if } |u| \le 1; \\
      \kappa_1^{-1} |f(x,u)|^{p'}&  \text{if } |u| \ge 1.      
    \end{cases}
    \label{Estim2}    
  \end{equation}
  where $p' := p/(p-1)$ is the conjugate exponent to $p$.
  Fix $\lambda$ and set $\Gamma := \{ x \in \IR^N : |u_\lambda(x)| \le 1 \}$.
  Inequality~(\ref{Estim1}) can be rewritten
  $$
    c_\lambda \ge  \kappa_2  \biggl(
    \int_\Gamma f(x,u_\lambda) u_\lambda  \intd x  +
    \int_{\IR^N \setminus\Gamma} f(x,u_\lambda) u_\lambda  \intd x  
    \biggr)$$
  Combining this with~(\ref{Estim2}), we get
  \begin{equation}
    \begin{split}
      f_0&  := \biggl(  \int_\Gamma  |f(x,u_\lambda)|^2 \intd x
      \biggr)^{1/2}  \le  \kappa_3 c_\lambda^{1/2}    \\
      f_\infty& := \biggl(  \int_{\IR^N\setminus\Gamma}  
      |f(x,u_\lambda)|^{p'} \intd x
      \biggr)^{1/p'}  \le  \kappa_3 c_\lambda^{1/p'}      
    \end{split}
    \label{Estim3}
  \end{equation}
  with $\kappa_3 := \max\bigl\{  (\kappa_1\kappa_2^{-1})^{1/2},\ 
  (\kappa_1\kappa_2^{-1})^{1/p'}  \bigr\}$. 
  Let us write $u_\lambda = y_\lambda + z_\lambda$ with $y_\lambda \in Y$,
  $z_\lambda \in Z$.
  Lemma~\ref{SplittingQ} and $\partial\E_\lambda(u_\lambda) =0$
  imply
  \begin{align*}
    \beta_\lambda \|z_\lambda\|^2 + \alpha_\lambda \|y_\lambda\|^2&
     \le  \Q_\lambda(z_\lambda) - \Q_\lambda(y_\lambda)  \\
    & =   {\textstyle{1\over 2}}
          \bigl\<  \partial\Q_\lambda(u_\lambda), z_\lambda-y_\lambda 
          \bigr\>  \\
    & =   {\textstyle{1\over 2}}
          \int f(x,u_\lambda) (z_\lambda - y_\lambda) \intd x\\
    & =   {\textstyle{1\over 2}}
          \int f(x,u_\lambda) u_\lambda \intd x - 
          \int f(x,u_\lambda) y_\lambda \intd x \\
    \intertext{and then, using~(\ref{Estim1}) and~(\ref{Estim3}),}
    \beta_\lambda \|z_\lambda\|^2 + \alpha_\lambda \|y_\lambda\|^2&
     \le \kappa_4 c_\lambda + \int_{\IR^N} |f(x,u_\lambda)| \, |y_\lambda| \\
    &\le \kappa_4 c_\lambda + f_0 |y_\lambda|_2 + f_\infty |y_\lambda|_p\\
    &\le \kappa_4 c_\lambda + 
         \kappa_5 (c_\lambda^{1/2} + c_\lambda^{1/p'}) \|y_\lambda\| \\
    &\le \kappa_4 c_\lambda + {\textstyle{2 \over \alpha_\lambda}}
         \kappa_5^2 (c_\lambda^{1/2} + c_\lambda^{1/p'})^2 +
         {\textstyle{\alpha_\lambda \over 2}} \|y_\lambda\|^2
  \end{align*}
  for some $\kappa_4, \kappa_5 >0$ independent of $\lambda$.
  Thus, moving ${1\over 2}\alpha_\lambda \|y_\lambda\|^2$ to the 
  left-hand side, we have
  $$
    {\textstyle{1\over 2}} N_\lambda \|u_\lambda\|^2  \le
    \kappa_4 c_\lambda  +  (2\kappa_5^2/\alpha_\lambda) 
    (c_\lambda^{1/2} + c_\lambda^{1/p'})^2.$$
  Now let $\lambda \to b$.  Therefore $\alpha_\lambda = \alpha_0$,
  $c_\lambda \to 0$, so that $c_\lambda^{1/p'} = \Order(c_\lambda^{1/2})$
  and
  $$
    N_\lambda  \|u_\lambda\|^2 = \Order(c_\lambda).$$
  The second estimate of $\|u_\lambda\|$ is obtained by plugging
  the estimate of $c_\lambda$ of 
  proposition~\ref{clambdaEstimations} into the first one and using the fact
  $\lim_{\lambda \to b}  (b-\lambda)/N_\lambda < +\infty$.
  The positivity of $1/(\beta-2) - N/4$ is equivalent to
  $\beta < 2 + 4/N$.
\end{proof}

\appendix
\section{Spectral decomposition  of $-\Delta + V$}
\label{section-spectral-Lp}

In this appendix, we will show that the splitting $Y \oplus Z$
of $H^1(\IR^N)$ introduced in section~\ref{section-spectral-gap} 
remains a direct
sum in the $\Leb^p(\IR^N)$'s  for $2 \le p \le \2$, in the sense
that $\cl_{\Leb^p} Y \cap \cl_{\Leb^p} Z = \{0\}$ with $\cl_{\Leb^p}$
denoting the closure in $\Leb^p(\IR^N)$.
We will start with the following stronger proposition.

\begin{prop}
\label{LebCountinuityProj}
  Let $H^1(\IR^N) = Y \oplus Z$ where $Y$ (resp.\ $Z$) is the
  negative (resp.\ positive) eigenspace of $D$ in $H^1$,
  and $P : H^1 \to H^1$ (resp.\ $Q= \id - P$) be the projector
  onto $Y$ (resp.\ $Z$) parallel to $Z$ (resp.\ $Y$).
  Then, for any $p \in [1,+\infty]$, the restrictions of
  $P$ and $Q$ to $H^1 \cap \Leb^p$ range in $H^1 \cap \Leb^p$
  and are  $\Leb^p$-continuous.
\end{prop}

\begin{proof}
  If $\tilde P, \tilde Q : \Leb^2(\IR^N) \to \Leb^2(\IR^N)$
  denote the projectors on the negative and positive eigenspaces of
  $D$ in $\Leb^2$ respectively,  it is well known 
  (see e.g.~\cite{Charles}, section~8) that 
  $P = \tilde P \restrict_{H^1}$ and $Q = \tilde Q \restrict_{H^1}$.
  So it is sufficient to prove the proposition for 
  $\Leb^2(\IR^N)$, $\tilde P$, and $\tilde Q$ instead of
  $H^1(\IR^N)$,   $P$, and $Q$.

  Denote $\Leb^p(\IR^N; \IC) = \Leb^p(\IR^N) + i \Leb^p(\IR^N)$ the
  complexification of $\Leb^p(\IR^N)$ and let $D_p$ be the
  operator
  \begin{align*}
    &D_p : \Leb^p(\IR^N; \IC) \to \Leb^p(\IR^N; \IC) : 
           u \mapsto -\Delta u + V(x) u\cr
    &\Dom(D_p) := \bigl\{ 
                 u \in \Leb^p(\IR^N; \IC) : D_p u \in \Leb^p(\IR^N; \IC)
                 \bigr\}.    
  \end{align*}
  It is proven in~\cite{Hempel} that the spectrum $\sigma(D_p) \subset \IR$
  is independent of $p \in [1,+\infty]$ and moreover, for any
  $\lambda \notin \sigma(D_p) = \sigma(D_2)= \sigma(D)$,
  \begin{equation}
    (D_p - \lambda)^{-1} = (D_2 - \lambda)^{-1}
    \quad\hbox{on } \Leb^p(\IR^N;\IC) \cap \Leb^2(\IR^N;\IC).
    \label{Hempel}
  \end{equation}
  Then $0 \notin \sigma(D_p)$ and we may speak of the
  (eigen)projectors $P_p$, $Q_p$ on the negative and positive
  eigenspaces of $D_p$.  Since $\sigma(D_p)$ is bounded 
  below, the projector $P_p$ may be defined as follows:
  if $\Gamma$ is a right-oriented curve around the negative part
  of $\sigma(D_p)$ (but not crossing the spectrum), then
  (see~\cite{Kato}):
  $$
    P_p = {\textstyle{1 \over 2 i \pi}}  \int_\Gamma
          (D_p - \lambda)^{-1} \intd \lambda.$$
  Accordingly, (\ref{Hempel}) yield
  $$
    P_p = P_2
    \quad\hbox{on } \Leb^p(\IR^N;\IC) \cap \Leb^2(\IR^N;\IC).$$
  That concludes the proof because $\tilde P = P_2 \restrict_{
  \Leb^2(\IR^N)}$ (and $\tilde Q = \id - \tilde P$).
\end{proof}

\begin{corollary}
\label{Lp-direct-sum}
  Let $Y\oplus Z$ be the splitting of $H^1(\IR^N)$ according
  to the positive and negative part of $\sigma(D)$.  Then, for
  all $p \in [2,\2]$,
  $$
    \Leb^p(\IR^N) = \cl_{\Leb^p} Y  \oplus  \cl_{\Leb^p} Z.$$
  \end{corollary}
  
\begin{proof}
  Let $P$ and $Q$ be the projectors  of proposition~\ref{LebCountinuityProj}.
  Since $P$, $Q$ are $\Leb^p$-continuous and $\D(\IR^N) \subset Y + Z$
  is dense in $\Leb^p(\IR^N)$, $P$ and $Q$ extend to continuous
  projectors $P_p$ and $Q_p$ on $\Leb^p(\IR^N)$ which leave invariant
  $\cl_{\Leb^p}Y$ and $\cl_{\Leb^p}Z$ respectively.
  From $PZ = \{0\}$ we infer $P_p (\cl_{\Leb^p}Z) = \{0\}$.
  Thus 
  $$
    \cl_{\Leb^p} Y  \cap \cl_{\Leb^p} Z = \{0\}.$$
  Now let $u \in \Leb^p(\IR^N)$.  By density there exists a sequence
  $(u_n) \subset \D(\IR^N)$ such that $u_n \to u$ in $\Leb^p$.
  By continuity,
  $$
    P_p u_n  \xrightarrow{\Leb^p}  P_p u  \in \cl_{\Leb^p} Y,
    \qquad
    Q_p u_n  \xrightarrow{\Leb^p}  Q_p u  \in \cl_{\Leb^p} Z,$$
  and so $u = P_p u + Q_p u \in  \cl_{\Leb^p} Y  + \cl_{\Leb^p} Z$.
\end{proof}

\section{Existence of a convex lower bound}
\label{Existenceof H}

This appendix is devoted to some elementary calculus showing
the existence of a convex lower bound of $\min\{ |u|^\beta, |u|^\alpha \}$
with the same asymptotic behavior.

\penalty-1000
\begin{lem}
  \label{ConvexLowerBound}
  Let $\beta \ge \alpha > 2$. There exists an even function 
  $H \in\C^1\bigl(\IR; [0,+\infty)\bigr)$ such that
  \begin{enumerate}[label={(\roman{*})}]
  \item\label{upper-H}  for all $u \in \IR$, 
    $H(u) \le \min\{ |u|^\beta, |u|^\alpha \}$\textup;
  \item\label{convex-H} $H$ is convex;
  \item\label{H-near0} $\lim_{u \to 0}        H(u)\, |u|^{-\beta} = 1$\textup;
  \item $\lim_{|u| \to \infty} H(u)\, |u|^{-\alpha} =1$\textup;
  \item\label{pseudo-homog} for all $u \in \IR$ and $t \ge 0$,
    $\min\{ t^\alpha, t^\beta\} H(u)
    \le H(tu)
    \le \max\{ t^\alpha, t^\beta\} H(u)$.
  \end{enumerate}
\end{lem}

\begin{rem}
  As consequences of the above facts, we get
  $H(0) = 0$, $\partial H(0)=0$, and $H(u) > 0$ for all $u\neq 0$.
\end{rem}

\begin{proof}
  Let $h \in \C(\IR;\IR)$ be the odd function defined by
  $
    h(u) := \min\bigl\{ \beta  |u|^{\beta-1},\penalty-2000
                        \alpha |u|^{\alpha-1} \bigr\}$
  for $u \ge 0$ and $H(u) := \int_0^u h$.
  The map $G(u) := \min\{ |u|^\beta, |u|^\alpha \}$ is $\C^1$ on
  $\IR\setminus\{ \pm 1\}$.  Let $\partial G$ be its derivative.
  It is clear that $G(u) = \int_0^u \partial G$ for all $u \in \IR$.
  Then (\ref{upper-H}) follows from $h(u) \le \partial G(u)$
  for all $u \in [0, +\infty) \setminus\{1\}$ and the evenness of $H$.
  Since $h$ is increasing, $H$ is (strictly) convex.
  An immediate  computation shows
  \begin{equation*}
    H(u) =
    \begin{cases}
      |u|^\beta&  \text{if } |u| \le \rho,\\
      \kappa + |u|^\alpha&  \text{otherwise},
    \end{cases}
  \end{equation*}
  for some $0< \rho \le 1$ and $\kappa := \rho^\beta -\rho^\alpha$.
  That proves the asymtotic behaviors of~$H$.
  Finally, the definition of $h$ implies that, for $t \ge 0$ and $u \ge 0$,
  $$
    \min\{ t^{\alpha-1}, t^{\beta-1} \} h(u)
    \le h(tu)
    \le \max\{ t^{\alpha-1}, t^{\beta-1} \} h(u)$$
  and then (\ref{pseudo-homog}) follows
  by integrating and taking into account the evenness of~$H$.
\end{proof}

\noindent
\textit{Acknowledgments.}\quad
The author is grateful to 
C.~A.~Stuart for fruitful
discussions.

\end{document}